\documentclass{amsart}

\newtheorem{theorem}[equation]{Theorem}
\newtheorem{lemma}[equation]{Lemma}
\newtheorem{corollary}[equation]{Corollary}
\newtheorem{proposition}[equation]{Proposition}

\numberwithin{equation}{section}

\usepackage{amscd,amsmath,amssymb,verbatim}

\begin{document}

\title[$A$-hypergeometric series associated to a lattice polytope]{$A$-hypergeometric series
 associated to a lattice polytope with a unique interior lattice point}
\author{Alan Adolphson}
\address{Department of Mathematics\\
Oklahoma State University\\
Stillwater, Oklahoma 74078}
\email{adolphs@math.okstate.edu}
\author{Steven Sperber}
\address{School of Mathematics\\
University of Minnesota\\
Minneapolis, Minnesota 55455}
\email{sperber@math.umn.edu}
\date{\today}
\begin{abstract}
We associate to lattice points ${\bf a}_0,{\bf a}_1,\dots,{\bf a}_N$ in ${\mathbb Z}^n$ an
 $A$-hyper\-geometric series $\Phi(\lambda_0,\dots,\lambda_N)$ with integer coefficients.  If
 ${\bf a}_0$ is the unique interior lattice point of the convex hull of ${\bf a}_1,\dots,{\bf a}_N$,
 then for every prime $p\neq 2$ the ratio $\Phi(\lambda)/\Phi(\lambda^p)$ has a $p$-adic analytic
 continuation to a closed unit polydisk minus a neighborhood of a hypersurface.  
\end{abstract}
\maketitle

\section{Introduction}

Let ${\bf a}_0,{\bf a}_1,\dots,{\bf a}_N\in{\mathbb Z}^n$ and put ${\bf a}_j = (a_{j1},\dots,a_{jn})$.  
For each $j=0,\dots,N$, let $\hat{\bf a}_j = (1,{\bf a}_j)\in{\mathbb Z}^{n+1}$ and put 
$A=\{\hat{\bf a}_j\}_{j=0}^N$.  We let $x_0,\dots,x_n$ be the coordinates on~${\mathbb R}^{n+1}$, so 
that $\hat{\bf a}_{j0} = 1$ while $\hat{\bf a}_{jk} = {\bf a}_{jk}$ for $k=1,\dots,n$.  Let $L$ be the 
module of relations among the $\hat{\bf a}_j$:
\[ L=\bigg\{l=(l_0,\dots,l_N)\in{\mathbb Z}^{N+1}\mid \sum_{j=0}^N l_j\hat{\bf a}_j = {\bf 0}\bigg\}. \]
We consider the $A$-hypergeometric system with parameter $-\hat{\bf a}_0$.  This is the system of 
partial differential equations in variables $\lambda_0,\dots,\lambda_N$ consisting of the operators
\[ \Box_l = \prod_{l_i>0} \bigg(\frac{\partial}{\partial \lambda_i}\bigg)^{l_i} - \prod_{l_i<0} 
\bigg(\frac{\partial}{\partial \lambda_i}\bigg)^{-l_i} \]
for $l\in L$ and the operators
\[ Z_i = \begin{cases} \sum_{j=0}^N a_{ji}\lambda_j\frac{\partial}{\partial \lambda_j} + a_{0i} & 
\text{for $i=1,\dots,n$,} \\
\sum_{j=0}^N \lambda_j\frac{\partial}{\partial \lambda_j} + 1 & \text{for $i=0$.} \end{cases} \]

Let $L_+ = \{ l\in L\mid l_i\geq 0\;\text{for $i=1,\dots,N$}\}.$  Taking $v=(-1,0,\dots,0)$ in 
\cite[Eq.~(3.36)]{SST} and applying \cite[Proposition 3.4.13]{SST} shows that this system has a 
formal solution $\lambda_0^{-1}\Phi(\lambda)$, where
\begin{equation}
\Phi(\lambda) = \sum_{l\in L_+} \frac{(-1)^{-l_0}(-l_0)!}{l_1!\cdots l_N!}\prod_{i=0}^N \lambda_i^{l_i}.
\end{equation}
Since $\sum_{i=0}^N l_i=0$ for $l\in L$, we can rewrite this as 
\begin{equation}
\Phi(\lambda) = \sum_{l\in L_+} \frac{(-1)^{\sum_{i=1}^N l_i} (l_1+\cdots+l_N)!}{l_1!\cdots l_N!} 
\prod_{i=1}^N \bigg(\frac{\lambda_i}{\lambda_0}\bigg)^{l_i},
\end{equation}
which shows that the coefficients of this series lie in ${\mathbb Z}$.  This implies that for each 
prime number $p$, the series $\Phi(\lambda)$ converges $p$-adically on the set 
\[ {\mathcal D} = \{(\lambda_0,\dots,\lambda_N)\in\Omega^{N+1}\mid  |\lambda_i/\lambda_0|<1 
\text{ for }i=1,\dots,N\}, \]
 (where $\Omega = $ completion of an algebraic closure of ${\mathbb Q}_p$) and $\Phi(\lambda)$ takes 
unit values there.  In particular, the ratio $\Phi(\lambda)/\Phi(\lambda^p)$ is an analytic function 
on ${\mathcal D}$ and takes unit values there.  

{\bf Remark:} Rewrite the series $\Phi(\lambda)$ according to powers of $-\lambda_0$:
\begin{equation}
\Phi(\lambda) = \sum_{k=0}^\infty \bigg(\sum_{\substack{l_1,\dots,l_N\in{\mathbb Z}_{\geq 0}\\ 
l_1\hat{\bf a}_1+\cdots+l_N\hat{\bf a}_N = k\hat{\bf a}_0}} \binom{k}{l_1,\dots,l_N}\lambda_1^{l_1}\cdots
\lambda_N^{l_N}\bigg)(-\lambda_0)^{-k}. 
\end{equation}
It is easy to see that the coefficient of $(-\lambda_0)^{-k}$ in this expression is the coefficient 
of $x^{k{\bf a}_0}$ in the Laurent polynomial $\big(\sum_{i=1}^N \lambda_ix^{{\bf a}_i}\big)^k$.  In 
particular, if ${\bf a}_0 = 0$, then it is the constant term of this Laurent polynomial.  The series 
$\Phi(\lambda)$ may thus be specialized to the hypergeometric series considered in Samol and van 
Straten\cite{SvS}.

Define a truncation of $\Phi(\lambda)$ by 
\[ \Phi_1(\lambda) = \sum_{\substack{l\in L_+\\ l_1+\cdots+l_N\leq p-1}} \frac{(-1)^{\sum_{i=1}^N l_i} 
(l_1+\cdots+l_N)!}{l_1!\cdots l_N!} \prod_{i=1}^N \bigg(\frac{\lambda_i}{\lambda_0}\bigg)^{l_i} \]
and let
\[ {\mathcal D}_+ = \{(\lambda_0,\dots,\lambda_N)\mid |\lambda_i/\lambda_0|\leq 1\text{ for } i=1,
\dots,N \text{ and }
|\Phi_1(\lambda)|= 1\}. \]
Note that ${\mathcal D}_+$ properly contains ${\mathcal D}$.
Let~$\Delta$ be the convex hull of the set $\{{\bf a}_1,\dots,{\bf a}_N\}$.  Our main result is the 
following theorem.

\begin{theorem}
Suppose that ${\bf a}_0$ is the unique interior lattice point of $\Delta$.  Then for every prime 
number $p\neq 2$ the ratio $\Phi(\lambda)/\Phi(\lambda^p)$ extends to an analytic function 
on~${\mathcal D}_+$.  
\end{theorem}

{\bf Remark:} If we specialize $\lambda_1,\dots,\lambda_N$ to elements of $\Omega$ of absolute value 
$\leq 1$, Equation~(1.3) allows us to regard $\Phi$ as a function of $t=-1/\lambda_0$ for $|t|<1$.  
By Theorem~1.4, this function of $t$ continues analytically to the region where $|t|\leq 1$ and 
$|\Phi_1(t)| =1$ (for $p\neq 2$).  When $\lambda_1,\dots,\lambda_N\in{\mathbb Z}_p$, this result on 
analytic continuation was proved recently by Mellit and Vlasenko\cite{MV} for all primes $p$.  We 
believe that the restriction $p\neq 2$ in Theorem 1.4 is an artifact of our method and that that 
result is in fact true for all $p$.  

{\bf Example:}(the Dwork family):  Let $N=n$ and let ${\bf a}_i = (0,\dots,n,\dots,0)$, where the 
$n$ occurs in the $i$-th position, for $i=1,\dots,n$.  Let ${\bf a}_0 = (1,\dots,1)$.  Then 
\[ L_+=\{(-nl,l,\dots,l,)\mid l\in{\mathbb Z}_{\geq 0}\} \]
and
\[ \Phi(\lambda) = \sum_{l=0}^{\infty} \frac{(-1)^{nl}(nl)!}{(l!)^n} \bigg(\frac{\lambda_1\cdots
\lambda_n}{\lambda_0^n}\bigg)^l. \]
Then Theorem 1.4 implies that for every prime $p\neq 2$ the ratio $\Phi(\lambda)/\Phi(\lambda^p)$ 
extends to an analytic function on ${\mathcal D}_+$.  We note that J.-D. Yu\cite{Y} has given a 
treatment of analytic continuation for the Dwork family using a more cohomological approach.

In \cite{D2}, Dwork gave a contraction mapping argument to prove $p$-adic analytic continuation of a 
ratio $G(\lambda)/G(\lambda^p)$ of normalized Bessel functions.  In \cite{AS}, we modified Dwork's 
approach to avoid computations of $p$-adic cohomology for exponential sums, which allowed us to 
greatly extend his analytic continuation result.  The proof we give here follows the method of 
\cite{AS}.  

Special values of the ratio $\Phi(\lambda)/\Phi(\lambda^p)$ are related to a unit root of the zeta 
function of the hypersurface $\sum_{i=0}^N \lambda_ix^{{\bf a}_0} = 0$.  We hope to return to this 
connection in a future article.  We believe that the methods of this paper will extend to complete 
intersections as well.

For the remainder of this paper we assume that $p\neq 2$.  This hypothesis is needed in Section~2 to 
define the endomorphism $\alpha^*$ of the space $S$.

\section{Contraction mapping}

We begin by constructing a mapping $\beta$ on a certain space of formal series whose coefficients 
are $p$-adic analytic functions.  The hypothesis that ${\bf a}_0$ is the unique interior point of 
$\Delta$ will imply that $\beta$ is a contraction mapping.

Let $\Omega$ be the completion of an algebraic closure of ${\bf Q}_p$ and put 
\[ R = \bigg\{ \xi(\lambda) = \sum_{\nu\in({\bf Z}_{\geq 0})^N} c_\nu\bigg(\frac{\lambda_1}{\lambda_0}
\bigg)^{\nu_1}\cdots \bigg(\frac{\lambda_N}{\lambda_0}\bigg)^{\nu_N}\mid \text{$c_\nu\in\Omega$ and 
$\{|c_\nu|\}_\nu$ is bounded}\bigg\} \]
Let $R'$ be the set of functions on ${\mathcal D}_+$ that are uniform limits of sequences of 
rational functions in the $\lambda_i/\lambda_0$ that are defined on ${\mathcal D}_+$.
The series in the ring $R$ are convergent and bounded on ${\mathcal D}$ and $R'$ is a subring 
of $R$.  We define a norm on $R$ by setting
\[ |\xi| = \sup_{\lambda\in{\mathcal D}} |\xi(\lambda)|. \]
Note that for $\xi\in R'$ one has $\sup_{\lambda\in{\mathcal D}}|\xi(\lambda)| = 
\sup_{\lambda\in{\mathcal D}_+} |\xi(\lambda)|$.  Both $R$ and $R'$ are complete in this norm.

Let $C$ be the real cone generated by the elements of $A$, let $M=C\cap{\mathbb Z}A$, and let 
$M^\circ\subset M$ be the subset consisting of those points that do not lie on any face of~$C$.  Let 
$\pi^{p-1}=-p$ and let $S$ be the $\Omega$-vector space of formal series
\[ S = \bigg\{\xi(\lambda,x) = \sum_{\mu\in M^{\circ}} \xi_\mu(\lambda) ({\pi}\lambda_0)^{-\mu_0}x^{-\mu} \mid 
\text{$\xi_\mu(\lambda)\in R$ and $\{|\xi_\mu|\}_\mu$ is bounded}\bigg\}. \]
Let $S'$ be defined analogously with the condition ``$\xi_\mu(\lambda)\in R$'' being replaced by 
``$\xi_\mu(\lambda)\in R'$''.  Define a norm on $S$ by setting
\[ |\xi(\lambda,x)| = \sup_\mu\{|\xi_\mu|\}. \]
Both $S$ and $S'$ are complete under this norm.

Define $\theta(t) = \exp(\pi(t-t^p)) = \sum_{i=0}^\infty b_it^i$.  One has 
(Dwork\cite[Sec\-tion~4a)]{D1})
\begin{equation}
{\rm ord}\: b_i\geq \frac{i(p-1)}{p^2}.
\end{equation}
Let
\[ F(\lambda,x) = \prod_{i=0}^N\theta(\lambda_ix^{\hat{\bf a}_i}) = \sum_{\mu\in M} B_\mu(\lambda)x^\mu, \]
where
\[ B_\mu(\lambda) = \sum_{\nu\in({\bf Z}_{\geq 0})^{N+1}} B^{(\mu)}_\nu\lambda^\nu \]
with
\begin{equation}
B^{(\mu)}_\nu = \begin{cases} \prod_{i=0}^N b_{\nu_i}  & \text{if $\sum_{i=0}^N \nu_i\hat{\bf a}_i = \mu$,} 
\\ 0 & \text{if $\sum_{i=0}^N \nu_i\hat{\bf a}_i\neq\mu$.} \end{cases} 
\end{equation}
The equation $\sum_{i=0}^N \nu_i\hat{\bf a}_i = \mu$ has only finitely many solutions $\nu\in({\mathbb 
Z}_{\geq 0})^{N+1}$, so $B_\mu$ is a polynomial in the $\lambda_i$.  Furthermore, all solutions of this 
equation satisfy $\sum_{i=0}^N \nu_i = \mu_0$, so $B_\mu$ is homogeneous of degree $\mu_0$.  We thus have
\begin{equation}
B_\mu(\lambda_0,\dots,\lambda_N) = \lambda_0^{\mu_0}B_\mu(1,\lambda_1/\lambda_0,\dots,\lambda_N/
\lambda_0). 
\end{equation}
Let $\tilde{\pi}\in\Omega$ satisfy ${\rm ord}\;\tilde{\pi} = (p-1)/p^2$.

\begin{lemma}
One has $B_\mu(1,\lambda_1/\lambda_0,\dots,\lambda_N/\lambda_0)\in R'$ and 
\[ |B_\mu(1,\lambda_1/\lambda_0,\dots,\lambda_N/\lambda_0)\leq |\tilde{\pi}^{\mu_0}|. \]
\end{lemma}

\begin{proof}
The first assertion is clear since $B_\mu$ is a polynomial.  We have
\[ B_\mu(1,\lambda_1/\lambda_0,\dots,\lambda_N/\lambda_0)=\sum_{\nu\in({\bf Z}_{\geq 0})^{N+1}} 
B^{(\mu)}_\nu(\lambda_1/\lambda_0)^{\nu_1}\cdots(\lambda_N/\lambda_0)^{\nu_N} \]
Using (2.1) and (2.2) gives
\[ {\rm ord}\: B^{(\mu)}_\nu\geq\sum_{i=0}^N {\rm ord}\: b_{\nu_i}\geq \sum_{i=0}^N \frac{\nu_i(p-1)}{p^2}= 
\mu_0\frac{p-1}{p^2}, \]
which implies the second assertion of the lemma.
\end{proof}

Using (2.3) we write
\begin{equation}
F(\lambda,x) = \sum_{\mu\in M} {B}_\mu(1,\lambda_1/\lambda_0,\dots,\lambda_N/\lambda_0)\lambda_0^{\mu_0} 
x^\mu.
\end{equation}
Let
\[ \xi(\lambda,x) = \sum_{\nu\in M^\circ} \xi_\nu(\lambda)({\pi}\lambda_0)^{-\nu_0}x^{-\nu}\in S. \]
We claim that the product $F(\lambda,x)\xi(\lambda^p,x^p)$ is well-defined as a formal series 
in $x$.  Formally we have
\[ F(\lambda,x)\xi(\lambda^p,x^p) = \sum_{\rho\in{\mathbb Z}^{n+1}} \zeta_\rho(\lambda)\lambda_0^{-\rho_0}
x^{-\rho}, \]
where
\begin{equation}
\zeta_\rho(\lambda) = \sum_{\substack{\mu\in M,\nu\in M^\circ \\ \mu-p\nu = -\rho}} {\pi}^{-\nu_0}
{B}_\mu(1,\lambda_1/\lambda_0,\dots,\lambda_N/\lambda_0)\xi_\nu(\lambda^p). 
\end{equation}
By Lemma 2.4, we have
\begin{equation}
|{\pi}^{-\nu_0}B_\mu(1,\lambda_1/\lambda_0,\dots,\lambda_N/\lambda_0)\xi_\nu(\lambda^p)|\leq 
|\tilde{\pi}^{\mu_0}\pi^{-\nu_0}|\cdot|\xi(\lambda,x)|. 
\end{equation}
Since $\mu=p\nu-\rho$, we have
\begin{align}
{\rm ord}\;\tilde{\pi}^{\mu_0}\pi^{-\nu_0} &= (p\nu_0-\rho_0)\frac{p-1}{p^2} - \frac{\nu_0}{p-1} 
\nonumber \\ 
&= \nu_0\bigg(\frac{p-1}{p} - \frac{1}{p-1}\bigg) - \rho_0\frac{p-1}{p^2}.
\end{align}
Since $(p-1)/p-1/(p-1)>0$ (we are using here our hypothesis that $p\neq 2$), this shows that 
$\tilde{\pi}^{\mu_0}\pi^{-\nu_0}\to 0$ as $\nu\to\infty$, so the series (2.6) converges to an element 
of~$R$.  The same argument shows that if $\xi(\lambda,x)\in S'$, then the series (2.6) converges to 
an element of $R'$.
 
Let $\gamma^\circ$ be the the truncation map
\[ \gamma^\circ\bigg(\sum_{\rho\in{\mathbb Z}^{n+1}} \zeta_\rho(\lambda)\lambda_0^{-\rho_0}x^{-\rho}\bigg) = 
\sum_{\rho\in M^\circ} \zeta_\rho(\lambda)\lambda_0^{-\rho_0}x^{-\rho} \]
and define for $\xi(\lambda,x)\in S$
\begin{align*}
\alpha^*\big(\xi(\lambda,x)\big) &= \gamma^\circ\big(F(\lambda,x)\xi(\lambda^p,x^p)\big) \\
 &= \sum_{\rho\in M^\circ}\zeta_\rho(\lambda)\lambda_0^{-\rho_0}x^{-\rho}. 
\end{align*}
For $\rho\in M^\circ$ put $\eta_\rho(\lambda) = {\pi}^{\rho_0}\zeta_\rho(\lambda)$, so that
\begin{equation}
\alpha^*(\xi(\lambda,x)) = \sum_{\rho\in M^\circ} \eta_\rho(\lambda)({\pi}\lambda_0)^{-\rho_0}x^{-\rho} 
\end{equation}
with (by (2.6))
\begin{equation}
\eta_\rho(\lambda) = \sum_{\substack{\mu\in M,\nu\in M^\circ\\ \mu-p\nu = -\rho}} {\pi}^{\rho_0-\nu_0}
{B}_\mu(1,\lambda_1/\lambda_0,\dots,\lambda_N/\lambda_0)\xi_\nu(\lambda^p). 
\end{equation}

\begin{proposition}
The map $\alpha^*$ is an endomorphism of $S$ and of $S'$, and for $\xi(\lambda,x)\in S$ we have
\begin{equation}
|\alpha^*(\xi(\lambda,x))|\leq |p|\cdot|\xi(\lambda,x)|.
\end{equation}
\end{proposition}

\begin{proof}
By (2.9), the proposition follows from the estimate 
\[ |\eta_\rho(\lambda)|\leq |p|\cdot|\xi(\lambda,x)| \quad\text{for all $\rho\in M^\circ$.} \] 
By (2.10), this estimate will follow from the estimate
\begin{equation}
|{\pi}^{\rho_0-\nu_0}{B}_\mu(1,\lambda_1/\lambda_0,\dots,\lambda_N/\lambda_0)|\leq |p|
\end{equation}
for all $\mu\in M$, $\nu\in M^\circ$, $\mu-p\nu=-\rho$.

Consider first the case $\mu_0\leq 2p-1$.  
For $0\leq i\leq p-1$ we have $b_i=\pi^i/i!$, hence $|b_i| = |\pi^i|$.  One checks easily that for 
$p\leq i\leq 2p-1$ one has $|b_i|\leq |\pi^i|$.  This implies by (2.2) that 
$|B_\mu(\lambda)|\leq |\pi^{\mu_0}|$, thus
\[ |{\pi}^{\rho_0-\nu_0}{B}_\mu(1,\lambda_1/\lambda_0,\dots,\lambda_N/\lambda_0)|\leq 
|\pi^{\rho_0-\nu_0+\mu_0}|. \]
Since $\mu=p\nu-\rho$ we have $\rho_0-\nu_0+\mu_0 = (p-1)\nu_0$, so
\begin{equation}
|{\pi}^{\rho_0-\nu_0}{B}_\mu(1,\lambda_1/\lambda_0,\dots,\lambda_N/\lambda_0)|\leq |p|^{\nu_0}. 
\end{equation}
Since $\nu_0\geq 1$ for $\nu\in M^\circ$, Eq.~(2.14) implies (2.13) for $\mu_0\leq 2p-1$.

Now consider the case $\mu_0\geq 2p$.  Lemma 2.4 implies that 
\begin{equation}
|{\pi}^{\rho_0-\nu_0}B_\mu(1,\lambda_1/\lambda_0,\dots,\lambda_N/\lambda_0)|\leq 
|\tilde{\pi}^{\mu_0}\pi^{\rho_0-\nu_0}|.
\end{equation}
We have (using $\mu=p\nu-\rho$)
\begin{align*}
{\rm ord}\;\tilde{\pi}^{\mu_0}\pi^{\rho_0-\nu_0} &= (p\nu_0-\rho_0)\frac{p-1}{p^2} + (\rho_0-\nu_0)
\frac{1}{p-1} \\
&= \nu_0\bigg(\frac{p-1}{p}-\frac{1}{p-1}\bigg) + \rho_0\bigg(\frac{1}{p-1}-\frac{p-1}{p^2}\bigg). 
\end{align*}
Since $\rho\in M^\circ$ we have $\rho_0\geq 1$, and since $\mu_0\geq 2p$ and $\mu=p\nu-\rho$ we must 
have $\nu_0\geq 3$.  It follows that 
\begin{equation}
\nu_0\bigg(\frac{p-1}{p}-\frac{1}{p-1}\bigg) + \rho_0\bigg(\frac{1}{p-1}-\frac{p-1}{p^2}\bigg)\geq 
3\bigg(\frac{p-1}{p}-\frac{1}{p-1}\bigg) + \bigg(\frac{1}{p-1}-\frac{p-1}{p^2}\bigg). 
\end{equation}
This latter expression is $>1$ for $p\geq 5$, hence (2.15) implies (2.13) when $\mu_0\geq 2p$ and 
$p\geq 5$.  

Finally, suppose that $p=3$ and $\mu_0\geq 2p = 6$.  By explicitly computing $b_i$, one checks that 
$|b_i| = |\pi^i|$ for $i\leq 8$.  This implies by (2.2) that $|B_\mu(\lambda)|\leq |\pi^{\mu_0}|$ for 
$\mu_0\leq 8$, so~(2.14) holds in this case and we conclude as before that~(2.13) holds also.  For 
$i=9,10$, one has $|b_i| = |\pi^{i-4}|$, so $|B_\mu(\lambda)|\leq |\pi^{\mu_0-4}|$ for $\mu_0=9,10$.  We 
thus have 
\[ |{\pi}^{\rho_0-\nu_0}{B}_\mu(1,\lambda_1/\lambda_0,\dots,\lambda_N/\lambda_0)|\leq 
|\pi^{\rho_0-\nu_0+\mu_0-4}| \]
for $\mu_0=9,10$.  Since $\mu=p\nu-\rho$, this gives
\[ |{\pi}^{\rho_0-\nu_0}{B}_\mu(1,\lambda_1/\lambda_0,\dots,\lambda_N/\lambda_0)|\leq |3^{\nu_0-2}|. \]
For $\mu_0=9,10$, $\mu=p\nu-\rho$ implies that $\nu_0\geq 4$, hence $|3^{\nu_0-2}|\leq |3^2|$ 
and~(2.13) holds.
One computes that $|b_{11}| = |\pi^9|$, so $|B_\mu(\lambda)|\leq |\pi^9|$ for $\mu_0=11$.  This gives 
(using $\mu=p\nu-\rho$) 
\[ |{\pi}^{\rho_0-\nu_0}{B}_\mu(1,\lambda_1/\lambda_0,\dots,\lambda_N/\lambda_0)|\leq |3^{\nu_0-1}| = |3^3| \]
since $\nu_0\geq 4$ for $\mu_0=11$, so (2.13) holds in this case.  Finally, if $\mu_0\geq 12$, then 
$\mu=p\nu-\rho$ implies $\nu_0\geq 5$.  When $p=3$, the left-hand side of (2.16) is $>1$ when 
$\nu_0\geq 5$, so~(2.15) implies~(2.13) in this case.
\end{proof}

{\bf Remark:} It follows from the proof of Proposition 2.11 that equality can hold in (2.13) only if 
$\nu_0=1$.  And since 
\[ {\pi}^{\rho_0-\nu_0}{B}_\mu(1,\lambda_1/\lambda_0,\dots,\lambda_N/\lambda_0)\in{\mathbb Q}_p(\pi,
\tilde{\pi})[\lambda_1/\lambda_0,\dots,\lambda_N/\lambda_0] \]
and ${\mathbb Q}_p(\pi,\tilde{\pi})$ is a discretely valued field, we conclude that there exists a 
rational number $C$, $0<C<1$, such that
\begin{equation}
|{\pi}^{\rho_0-\nu_0}{B}_\mu(1,\lambda_1/\lambda_0,\dots,\lambda_N/\lambda_0)|\leq C|p|
\end{equation}
for all $\mu\in M$, $\nu\in M^\circ$, $\mu-p\nu=-\rho$, with $\nu_0>1$.  

\begin{lemma}
Suppose ${\bf a}_0$ is the unique interior lattice point of $\Delta$.  If 
$\xi_{\hat{\bf a}_0}(\lambda)=0$, then $|\alpha^*(\xi(\lambda,x))|\leq C |p|\cdot|\xi(\lambda,x)|$.
\end{lemma}

\begin{proof}
Since ${\bf a}_0$ is the unique interior lattice point of $\Delta$, the point $\nu=\hat{\bf a}_0$ is 
the unique element of~$M^\circ$ with $\nu_0=1$.   So for $\nu\in M^\circ$, $\nu\neq\hat{\bf a}_0$, we 
have $\nu_0\geq 2$.  The assertion of the lemma then follows from (2.10) and (2.17).
\end{proof}

We examine the polynomial $B_{(p-1)\hat{\bf a}_0}(\lambda)$ to determine its relation to 
$\Phi_1(\lambda)$.  Let
\[ V = \bigg\{v=(v_0,\dots,v_N)\in ({\mathbb Z}_{\geq 0})^{N+1} \mid \sum_{i=0}^N v_i\hat{\bf a}_i = 
(p-1)\hat{\bf a}_0\bigg\}. \]
From (2.2) we have
\[ B_{(p-1)\hat{\bf a}_0}(\lambda) = \sum_{v\in V} \bigg(\prod_{i=0}^N b_{v_i}\bigg)\lambda_0^{v_0}\cdots
\lambda_N^{v_N}. \]
For $v\in V$ we have $\sum_{i=0}^N v_i = p-1$ so $v_0=(p-1)-v_1-\cdots-v_N$.  Furthermore, 
$v_i\leq p-1$ for all $i$ so $b_{v_i} = \pi^{v_i}/v_i!$.  And since $\pi^{p-1} = -p$, this implies
\[ B_{(p-1)\hat{\bf a}_0}(\lambda) = -p\lambda_0^{p-1}\sum_{v\in V}\frac{(\lambda_1/\lambda_0)^{v_1}\cdots 
(\lambda_N/\lambda_0)^{v_N}}{(p-1-v_1-\ldots-v_N)!v_1!\cdots v_N!}. \]
It follows that
\[ p^{-1}B_{(p-1){\bf a}_0}(1,\lambda_1/\lambda_0,\dots,\lambda_N/\lambda_0) = -\sum_{v\in V}
\frac{(\lambda_1/\lambda_0)^{v_1}\cdots (\lambda_N/\lambda_0)^{v_N}}
{(p-1-v_1-\ldots-v_N)!v_1!\cdots v_N!}, \]
 a polynomial in the $\lambda_i/\lambda_0$ with $p$-integral coefficients.

\begin{lemma}
$\Phi_1(\lambda)\equiv p^{-1}B_{(p-1){\bf a}_0}(1,\lambda_1/\lambda_0,\dots,\lambda_N/\lambda_0) \pmod{p}$.
\end{lemma}

\begin{proof}
The map $(v_0,\dots,v_N)\mapsto (-v_1-\dots-v_N,v_1,\dots,v_N)$ is a one-to-one correspondence from 
$V$ to the elements $l\in L_+$ satisfying $l_1+\cdots+l_N\leq p-1$.  The lemma then follows 
immediately from the congruence
\[ -\frac{1}{(p-1-m)!}\equiv (-1)^m m!\pmod{p} \quad\text{for $0\leq m\leq p-1$}, \]
which is implied by the congruence $(p-1)!\equiv -1\pmod{p}$.
\end{proof}

\begin{corollary} The polynomial $B_{(p-1){\bf a}_0}(1,\lambda_1/\lambda_0,\dots,\lambda_N/\lambda_0)$ is 
an invertible element of $R'$ with $|B_{(p-1){\bf a}_0}(1,\lambda_1/\lambda_0,\dots,\lambda_N/\lambda_0)| 
= |p|$.
\end{corollary}

\begin{proof}
The first assertion is an immediate consequence of Lemma 2.19.  The assertion about the norm follows 
from the fact that all the coefficients are divisible by~$p$ and the constant term equals $-p/(p-1)!$.
\end{proof}

Suppose that ${\bf a}_0$ is the unique interior lattice point of $\Delta$, so that 
$\nu=\hat{\bf a}_0$ is the unique element of $M^\circ$ with $\nu_0=1$.  From Equation (2.10) we have
\begin{equation}
\begin{split}
\eta_{\hat{\bf a}_0}(\lambda) &= \sum_{\substack{\mu\in M,\nu\in M^\circ\\ \mu=p\nu -\hat{\bf a}_0}} {\pi}^{1-\nu_0} 
B_\mu(1,\lambda_1/\lambda_0,\dots,\lambda_N/\lambda_0)\xi_\nu(\lambda^p) \\
 & = B_{(p-1)\hat{\bf a}_0}(1,\lambda_1/\lambda_0,\dots,\lambda_N/\lambda_0)\xi_{\hat{\bf a}_0}(\lambda^p)  \\
& \quad +\sum_{\substack{\mu\in M,\nu\in M^\circ\\ \mu=p\nu -\hat{\bf a}_0\\ \nu_0\geq 2}} {\pi}^{1-\nu_0} 
B_\mu(1,\lambda_1/\lambda_0,\dots,\lambda_N/\lambda_0)\xi_\nu(\lambda^p).
\end{split}
\end{equation}

\begin{lemma}
Suppose that ${\bf a}_0$ is the unique interior lattice point of $\Delta$.  
If $\xi_{\hat{\bf a}_0}(\lambda)$ is an invertible element of $R$ (resp.~$R'$) and 
$|\xi_{\hat{\bf a}_0}(\lambda)|= |\xi(\lambda,x)|$, then $\eta_{\hat{\bf a}_0}(\lambda)$ is also an 
invertible element of $R$ (resp.~$R'$) and $|\eta(\lambda,x)| = |\eta_{\hat{\bf a}_0}(\lambda)| = 
|p|\cdot |\xi_{\hat{\bf a}_0}(\lambda)|$.
\end{lemma}

\begin{proof}
By Corollary 2.20 we have that $B_{(p-1)\hat{\bf a}_0}(1,\lambda_1/\lambda_0,\dots,\lambda_N/\lambda_0)
\xi_{\hat{\bf a}_0}(\lambda^p)$ is an invertible element of norm $|p|\cdot|\xi_{\hat{\bf a}_0}(\lambda)|$.  
By hypothesis, we have
\[ |\xi_\nu(\lambda^p)|\leq |\xi_{\hat{\bf a}_0}(\lambda)|\quad\text{for all $\nu\in M^\circ$.} \]
Equation (2.17) with $\rho = \hat{\bf a}_0$ gives
\begin{equation}
|{\pi}^{1-\nu_0} B_\mu(1,\lambda_1/\lambda_0,\dots,\lambda_N/\lambda_0)|\leq C|p|
\end{equation}
for all $\mu\in M$, $\nu\in M^\circ$, $\mu-p\nu=-\hat{\bf a}_0$, $\nu_0\geq 2$ and some constant $C$, 
$0<C<1$.
 Equation (2.21) then implies that $\eta_{\hat{\bf a}_0}(\lambda)$ is invertible and that 
\[ |\eta_{\hat{\bf a}_0}(\lambda)|=|p|\cdot |\xi_{\hat{\bf a}_0}(\lambda)|. \]
Equation (2.12) then implies that $|\eta(\lambda,x)| = |p|\cdot|\xi_{\hat{\bf a}_0}(\lambda)|$.  
\end{proof}

Suppose that ${\bf a}_0$ is the unique interior lattice point of $\Delta$.  Put 
\[ T = \{\xi(\lambda,x)\in S\mid \text{$\xi_{\hat{\bf a}_0}(\lambda) = 1$ and $|\xi(\lambda,x)| = 1$}\} 
\]
and put $T' = T\cap S'$.  It follows from Lemma~2.22 that if $\xi(\lambda,x)\in T$, then 
$\eta_{\hat{\bf a}_0}(\lambda)$ is invertible.  We may thus define for $\xi(\lambda,x)\in T$
\[ \beta(\xi(\lambda,x)) = \frac{\alpha^*(\xi(\lambda,x))}{\eta_{\hat{\bf a}_0}(\lambda)}. \]
Lemma 2.22 also implies that 
\[ \bigg|\frac{\alpha^*(\xi(\lambda,x))}{\eta_{\hat{\bf a}_0}(\lambda)}\bigg|= 1, \]
so $\beta(T)\subseteq T$.  It is then clear that $\beta(T')\subseteq T'$.

\begin{proposition}
Suppose that ${\bf a}_0$ is the unique interior lattice point of $\Delta$.  Then the operator 
$\beta$ is a contraction mapping on the complete metric space $T$.  More precisely, for $C$ as in 
$(2.17)$, if $\xi^{(1)}(\lambda,x),\xi^{(2)}(\lambda,x)\in T$, then
\[ |\beta\big(\xi^{(1)}(\lambda,x)\big)-\beta\big(\xi^{(2)}(\lambda,x)\big)|\leq C|\xi^{(1)}(\lambda,x)-
\xi^{(2)}(\lambda,x)|. \]
\end{proposition}

\begin{proof}
We have (in the obvious notation)
\begin{equation*}
\begin{split}
\beta\big(\xi^{(1)}(\lambda,x)\big)-\beta\big(\xi^{(2)}(\lambda,x)\big) &= \frac{\alpha^*\big(
\xi^{(1)}(\lambda,x)\big)}{\eta^{(1)}_{\hat{\bf a}_0}(\lambda)} - \frac{\alpha^*\big(\xi^{(2)}(\lambda,x)
\big)}{\eta^{(2)}_{\hat{\bf a}_0}(\lambda)} \\
 &= \frac{\alpha^*\big(\xi^{(1)}(\lambda,x)-\xi^{(2)}(\lambda,x)\big)}{\eta^{(1)}_{\hat{\bf a}_0}(\lambda)}
 \\
 & \qquad - \alpha^*\big(\xi^{(2)}(\lambda,x)\big)\frac{\eta^{(1)}_{\hat{\bf a}_0}(\lambda) - 
\eta^{(2)}_{\hat{\bf a}_0}(\lambda)}{\eta^{(1)}_{\hat{\bf a}_0}(\lambda)\eta^{(2)}_{\hat{\bf a}_0}(\lambda)}.
\end{split}
\end{equation*}
By Lemmas 2.18 and 2.22 we have
\[ \bigg|\frac{\alpha^*(\xi^{(1)}(\lambda,x)-\xi^{(2)}(\lambda,x))}{\eta^{(1)}_{\hat{\bf a}_0}(\lambda)}
\bigg| \leq
C|\xi^{(1)}(\lambda,x)-\xi^{(2)}(\lambda,x))|. \]
Since $\eta^{(1)}_{\hat{\bf a}_0}(\lambda)-\eta^{(2)}_{\hat{\bf a}_0}(\lambda)$ is the coefficient of 
$x^{-\hat{\bf a}_0}$ in $\alpha^*\big(\xi^{(1)}(\lambda,x)-\xi^{(2)}(\lambda,x)\big)$, we have
\[ |\eta^{(1)}_{\hat{\bf a}_0}(\lambda)-\eta^{(2)}_{\hat{\bf a}_0}(\lambda)|\leq 
|\alpha^*\big(\xi^{(1)}(\lambda,x)-\xi^{(2)}(\lambda,x)\big)|\leq C |p|\cdot |\xi^{(1)}(\lambda,x)-
\xi^{(2)}(\lambda,x))| \]
by Lemma 2.18.  We have $|\eta^{(1)}_{\hat{\bf a}_0}(\lambda)\eta^{(2)}_{\hat{\bf a}_0}(\lambda)|=|p^2|$ by 
Lemma~2.22, so by (2.12)
\[ \bigg| \alpha^*\big(\xi^{(2)}(\lambda,x)\big)\frac{\eta^{(1)}_{\hat{\bf a}_0}(\lambda) - 
\eta^{(2)}_{\hat{\bf a}_0}(\lambda)}{\eta^{(1)}_{\hat{\bf a}_0}(\lambda)\eta^{(2)}_{\hat{\bf a}_0}(\lambda)}\bigg|
\leq
C|\xi^{(1)}(\lambda,x)-\xi^{(2)}(\lambda,x))|. \]
This establishes the proposition.
\end{proof}

By a well-known theorem, Proposition 2.24 implies that $\beta$ has a unique fixed point in $T$.  And 
since $\beta$ is stable on $T'$, that fixed point must lie in $T'$.  This fixed point of $\beta$ is 
related to a certain eigenvector of $\alpha^*$.  Suppose that $\xi(\lambda,x)\in S$ is an 
eigenvector of $\alpha^*$, say,
\[ \alpha^*\big(\xi(\lambda,x)\big) = \kappa\xi(\lambda,x), \]
with $\xi_{\hat{\bf a}_0}(\lambda)$ invertible and $|\xi(\lambda,x)| = |\xi_{\hat{\bf a}_0}(\lambda|$.  
Then $\xi(\lambda,x)/\xi_{\hat{\bf a}_0}(\lambda)\in T$.  
\begin{lemma}
With the above notation, $\xi(\lambda,x)/\xi_{\hat{\bf a}_0}(\lambda)$ is the unique fixed point 
of~$\beta$, hence $\xi(\lambda,x)/\xi_{\hat{\bf a}_0}(\lambda)\in T'$.  In particular, 
$\xi_\rho(\lambda)/\xi_{\hat{\bf a}_0}(\lambda)\in R'$ for all $\rho\in M^\circ$.
\end{lemma}

\begin{proof}
We have
\begin{equation}
\alpha^*\bigg(\frac{\xi(\lambda,x)}{\xi_{\hat{\bf a}_0}(\lambda)}\biggr) = \frac{\alpha^*\big(
\xi(\lambda,x)\big)}{\xi_{\hat{\bf a}_0}(\lambda^p)} = \bigg(\frac{\kappa\xi_{\hat{\bf a}_0}(\lambda)}
{\xi_{\hat{\bf a}_0}(\lambda^p)}\bigg) \frac{\xi(\lambda,x)}{\xi_{\hat{\bf a}_0}(\lambda)}.
\end{equation}
By the definition of $\beta$, this implies the result.
\end{proof}

\begin{corollary}
With the above notation, $\xi_{\hat{\bf a}_0}(\lambda)/\xi_{\hat{\bf a}_0}(\lambda^p) \in R'$.
\end{corollary}

\begin{proof}
Since $\alpha^*$ is stable on $S'$, Lemma 2.25 implies that the right-hand side of~(2.26) lies in 
$S'$.  Since the coefficient of $(\pi\lambda_0)^{-1}x^{-\hat{\bf a}_0}$ on the right-hand side of (2.26) 
is $\kappa\xi_{\hat{\bf a}_0}(\lambda)/\xi_{\hat{\bf a}_0}(\lambda^p)$, the result follows.
\end{proof}

In the next section we find the fixed point of $\beta$ by finding the corresponding eigenvector of 
$\alpha^*$.  This eigenvector 
will be constructed from solutions of the $A$-hypergeometric system; in particular, we shall have 
$\xi_{\hat{\bf a}_0}(\lambda) = \Phi(\lambda)$, so Corollary~2.27 will imply Theorem~1.4.

\section{Fixed point}

We begin with a preliminary calculation.  Consider the series
\[ G(\lambda_0,x) = \sum_{l=0}^{\infty} (-1)^l l!(\pi\lambda_0x^{\hat{\bf a}_0})^{-1-l}. \]
It satisfies the equations
\begin{equation}
\gamma_-\bigg(x_i\frac{\partial}{\partial x_i}-\pi\lambda_0\hat{\bf a}_{0i}x^{\hat{\bf a}_0}\bigg)
G(\lambda_0,x)  = 0
\end{equation}
for $i=0,1,\dots,n$, where $\gamma_-$ is the operator on series defined by
\[\gamma_-\bigg(\sum_{l=-\infty}^{\infty} c_l(\pi\lambda_0x^{\hat{\bf a}_0})^{-1-l}\bigg) = \sum_{l=0}^{\infty} 
c_l(\pi\lambda_0x^{\hat{\bf a}_0})^{-1-l}. \]
It is straightforward to check that the series $\sum_{l=0}^{\infty} c_l(\pi\lambda_0
x^{\hat{\bf a}_0})^{-1-l}$ that satisfy the operators $\gamma_-\circ(x_i\partial/\partial x_i-
\pi\lambda_0\hat{\bf a}_{0i}x^{\hat{\bf a}_0})$ form a one-dimensional space.

Consider the series 
\[ H(\lambda_0,x)=\gamma_-\big(\exp\big(\pi(\lambda_0x^{\hat{\bf a}_0}-\lambda_0^px^{p\hat{\bf a}_0})\big)
G(\lambda_0^p,x^p)\big). \]
Formally one has 
\[ x_i\frac{\partial}{\partial x_i}-\pi\lambda_0\hat{\bf a}_{0i}x^{\hat{\bf a}_0} = \exp(\pi\lambda_0
x^{\hat{\bf a}_0})\circ x_i\frac{\partial}{\partial x_i}\circ \frac{1}{\exp(\pi\lambda_0x^{\hat{\bf a}_0})} \]
and
\[ \gamma_-\circ\bigg(x_i\frac{\partial}{\partial x_i} - \pi\lambda_0\hat{\bf a}_{0i}x^{\hat{\bf a}_0}
\bigg)\circ \gamma_- = 
\gamma_-\circ\bigg(x_i\frac{\partial}{\partial x_i} - \pi\lambda_0\hat{\bf a}_{0i}x^{\hat{\bf a}_0}
\bigg). \]
It follows that
\begin{multline*}
\gamma_-\bigg(x_i\frac{\partial}{\partial x_i} - \pi\lambda_0\hat{\bf a}_{0i}x^{\hat{\bf a}_0}\bigg)
H(\lambda_0,x) = \\
\gamma_-\bigg(\exp(\pi\lambda_0x^{\hat{\bf a}_0})x_i\frac{\partial}{\partial x_i}\bigg(G(\lambda_0^p,x^p)/
\exp(\pi\lambda_0^px^{p\hat{\bf a}_0})\bigg)\bigg). 
\end{multline*}
Eq.~(3.1) implies that this latter expression equals 0, i.~e., $H(\lambda_0,x)$ also satisfies the 
operators $\gamma_-\circ(x_i\partial/\partial x_i-\pi\lambda_0\hat{\bf a}_{0i}x^{\hat{\bf a}_0})$.  Since 
the solutions of this operator form a one-dimensional space, we have $H(\lambda_0,x) = 
cG(\lambda_0,x)$ for some constant $c$.  

We determine the constant $c$ by comparing the coefficients of $(\pi\lambda_0x^{\hat{\bf a}_0})^{-p}$ in 
$G(\lambda_0,x)$ and $H(\lambda_0,x)$.  The coefficient of $(\pi\lambda_0x^{\hat{\bf a}_0})^{-p}$ in 
$G(\lambda_0,x)$ is $(-1)^{p-1}(p-1)!$.
A calculation shows that the coefficient of $(\pi\lambda_0x^{\hat{\bf a}_0})^{-p}$ in $H(\lambda_0,x)$ is
\[ -p\sum_{l=0}^\infty b_{pl}(-1)^ll!\pi^{-l}. \]
By Boyarsky\cite[Eq.~(5.6)]{B} (or see \cite[Lemma~3]{A}) this equals $-p\Gamma_p(p)=(-1)^{p+1} p!$, 
where $\Gamma_p$ denotes the $p$-adic gamma function.  It follows that the constant $c$ satisfies 
the equation
\[ (-1)^{p+1}p!=c(-1)^{p-1}(p-1)!, \]
so $c=p$ and we conclude that
\begin{equation}
H(\lambda_0,x) = p G(\lambda_0,x).
\end{equation}

We now consider the series
\[ \xi(\lambda,x) = \gamma^\circ\bigg(G(\lambda_0,x)\prod_{i=1}^N \exp(\pi\lambda_i
x^{\hat{\bf a}_i})\bigg), \]
where $\gamma^\circ$ is as defined in Section~2.  A calculation shows that
\begin{multline*}
G(\lambda_0,x)\prod_{i=1}^N \exp(\pi\lambda_ix^{\hat{a}_i}) = \\ 
\sum_{\rho\in{\mathbb Z}^{n+1}} \bigg( \sum_{\substack{l_0,\dots,l_N\in{\mathbb Z}_{\geq 0}\\ 
l_1\hat{\bf a}_1+\cdots+l_N\hat{\bf a}_N-(l_0+1)\hat{\bf a}_0 = -\rho}} \frac{(-1)^{l_0}l_0!
\pi^{-1-l_0+\sum_{i=1}^N l_i}
\lambda_1^{l_1}\cdots\lambda_N^{l_N}\lambda_0^{-l_0-1}}{l_1!\cdots l_N!}\bigg)x^{-\rho}. 
\end{multline*}
It follows that we can write $\xi(\lambda,x)$ as
\begin{equation}
\xi(\lambda,x) = \sum_{\rho\in M^\circ}  \xi_\rho(\lambda)(\pi\lambda_0)^{-\rho_0}x^{-\rho},
\end{equation}
where
\begin{multline}
\xi_\rho(\lambda) = \\ 
\sum_{\substack{l_0,\dots,l_N\in{\mathbb Z}_{\geq 0}\\ l_1\hat{\bf a}_1+\cdots+l_N\hat{\bf a}_N-(l_0+1)\hat{\bf a}_0 = -\rho}} 
\frac{(-1)^{\rho_0-1+\sum_{i=1}^N l_i}(\rho_0-1+\sum_{i=1}^N l_i)!}{l_1!\cdots l_N!} \prod_{i=1}^N 
\bigg(\frac{\lambda_i}{\lambda_0}\bigg)^{l_i}.
\end{multline}

{\bf Remark:}  One can check that the series $\lambda_0^{-\rho_0}\xi_\rho(\lambda)$ is a solution of 
the $A$-hyper\-geometric system with parameter $-\rho$ (although we shall not make use of that fact 
here).  Thus the coefficients of powers of $x$ in the series $\xi(\lambda,x)$ form a ($p$-adically 
normalized) family of ``contiguous'' $A$-hypergeometric functions.

The series $\xi_\rho(\lambda)$ have coefficients in ${\mathbb Z}$ for all $\rho\in M^\circ$, hence 
$\xi(\lambda,x)\in S$ and $|\xi(\lambda,x)|\leq 1$.  
Note that $\xi_{\hat{\bf a}_0}(\lambda) = \Phi(\lambda)$, where $\Phi(\lambda)$ is defined by (1.1).  
In particular, $\xi_{\hat{\bf a}_0}(\lambda)$ takes unit values on ${\mathcal D}$, hence 
$|\xi_{\hat{\bf a}_0}(\lambda)|=1$ and $\xi_{\hat{\bf a}_0}(\lambda)$ is an invertible element of $R$.  To 
show that $\xi(\lambda,x)$ satisfies the hypothesis of Lemma~2.25, it remains only to establish the 
following result.

\begin{lemma}
We have $\alpha^*\big(\xi(\lambda,x)\big) = p\xi(\lambda,x)$.
\end{lemma}

\begin{proof}
From the definitions we have
\[ \alpha^*(\xi(\lambda,x)) = \gamma^\circ\bigg(F(\lambda,x)\cdot\gamma^\circ\bigg(G(\lambda_0^p,x^p)
\prod_{i=1}^N \exp(\pi\lambda_i^px^{p\hat{\bf a}_i})\bigg)\bigg). \]
One checks that $\big(\text{mult.\ by $F(\lambda,x)$}\big)\circ\gamma^\circ = \gamma^\circ\circ 
\big(\text{mult.\ by $F(\lambda,x)$}\big)$.  Furthermore, by definition we have $F(\lambda,x) = 
\prod_{i=0}^N \exp(\pi\lambda_ix^{\hat{\bf a}_i})/\exp(\pi\lambda_i^px^{p\hat{\bf a}_i})$.  It follows that
\[ \alpha^*(\xi(\lambda,x)) = \gamma^\circ\bigg(\gamma^\circ\bigg(\frac{\exp(\pi
\lambda_0x^{\hat{\bf a}_0})}{\exp(\pi\lambda_0^px^{p\hat{\bf a}_0})}G(\lambda_0^p,x^p)\bigg)\prod_{i=1}^N 
\exp(\pi\lambda_ix^{\hat{\bf a}_i})\bigg). \]
By (3.2) we finally have
\[ \alpha^*(\xi(\lambda,x)) = \gamma^\circ\bigg(pG(\lambda_0,x)\prod_{i=1}^N \exp(\pi\lambda_i
x^{\hat{\bf a}_i})\bigg) = p\xi(\lambda,x). \]
\end{proof}

\begin{proof}[Proof of Theorem $1.4$]
We have shown that the series $\xi(\lambda,x)$ given by (3.3) and~(3.4) satisfies the hypotheses of 
Lemma 2.25.  Since $\xi_{\hat{\bf a}_0}(\lambda)=\Phi(\lambda)$, Theorem~1.4 follows from Corollary~2.27.
\end{proof}

\end{document}